\documentclass[preprint,12pt,authoryear]{elsarticle}

\usepackage{amssymb}
\usepackage{amsmath}

\usepackage{amsthm}

\newtheorem{theorem}{Theorem}
\newtheorem{lemma}{Lemma}
\newtheorem{proposition}{Proposition}

\usepackage{amsthm}
\usepackage{algorithm}
\usepackage{algpseudocode}
\usepackage{microtype}

\biboptions{authoryear}   
\journal{Nuclear Physics B}

\begin{document}

\begin{frontmatter}



\title{Explainable Artificial Intelligence for Financial Integral Equations: A Fixed-Point Neural Operator Approach} 


\author[1]{Sanjay Kumar Mohanty\corref{cor1}}
\ead{sanjaymath@gmail.com}
\cortext[cor1]{Corresponding author}
\affiliation{ organization={Department of Mathematics, SAS, Vellore Institute of Technology, Vellore 632014, India},
            }

\begin{abstract}
The explainable artificial intelligence is used to analyze the stochastic Fredholm integral equations (SFIEs) and stochastic
deep neural networks (SDNNs). The neural operator-based stochastic fixed point framework is used to develop SDNNs. The solution of an SFIE is obtained
through successive applications of an integral operator, and this iterative structure naturally
resembles the layered architecture of a neural network. The associated nonlinear versions of SFIE and SDNN are discussed. The SFIE and SDNN are used to solve the Black Scholes equation, contagion dynamics of financial networks, and the Merten jump diffusion equation. It is observed that the results obtained through SFIE and SDNN for all the applications agree well. 
\end{abstract}




\end{frontmatter}

\section{Introduction}
Stochastic Fredholm integral equations arise as a natural extension of classical integral equation theory when randomness is incorporated into the kernel, forcing terms, or solution process, making them suitable for modeling systems with nonlocal interactions under uncertainty. Building on the foundational work of Erik Ivar Fredholm \cite{fredholm1903}, early studies focused on deterministic formulations, which were later generalized using stochastic analysis inspired by Kiyosi Itô \cite{ito1944}. Over time, researchers developed theoretical results on existence, uniqueness, and stability, alongside analytical and numerical methods for solving such equations. Classical contributions in the 1970s and 1980s addressed stochastic integral equations in general, while more specialized work on stochastic Fredholm and Volterra–Fredholm types emerged in the late 20th and early 21st centuries. In recent years, the field has seen significant progress with advanced computational techniques, including wavelet-based and operational matrix methods, as well as modern machine learning approaches. For instance, recent studies have proposed efficient numerical schemes for stochastic Volterra–Fredholm equations with delay and improved convergence properties \cite{yao2024} , while deep learning-based frameworks such as multi-grade neural networks have been introduced to solve highly oscillatory Fredholm equations with strong accuracy and stability guarantees. These developments highlight a clear transition from classical analytical approaches to data-driven and computationally efficient methods, reflecting the growing importance of stochastic Fredholm integral equations in modeling complex systems across science and engineering \cite{jiang2026}.

Stochastic Differential Equations (SDEs) offer a natural and powerful way to describe systems that evolve under uncertainty, extending classical differential equations by incorporating randomness through processes such as the Wiener process. The origins of this field can be traced to the physical insights of Paul Langevin \cite{langevin1908}, who modeled random motion in fluids, and were later placed on firm mathematical footing by Kiyosi Itô \cite{ito1944} through the development of Itô calculus, which made it possible to rigorously analyze stochastic systems. The formal construction of Brownian motion by Norbert Wiener \cite{wiener1923} further strengthened the theoretical backbone of SDEs. Over the years, these equations have moved beyond theory to become essential tools in many applied fields, most notably in finance, where the work of Fischer Black and Myron Scholes \cite{blackscholes1973} demonstrated how asset prices could be modeled under uncertainty. Recently, SDEs continue to evolve, with ongoing research focusing on efficient numerical methods and new applications in complex, data-driven environments, reflecting their central role in understanding real-world systems influenced by randomness \cite{higham2001}.

Explainable Artificial Intelligence (XAI) based neural operator approach has recently become central in financial mathematics, enabling transparency in complex models used for credit risk, fraud detection, and financial distress prediction while maintaining high predictive performance. Modern approaches integrate interpretability techniques such as SHAP and LIME into machine learning frameworks to provide both local and global explanations, improving trust, fairness, and regulatory compliance in financial decision-making. These developments reflect a shift toward combining accuracy with accountability, making XAI an essential component of robust and interpretable financial modeling. Kyriakos \cite{kyriakos} discussed the family of explainable machine learning techniques and presented the Fredholm neural networks (Fredholm NNs): deep neural networks (DNNs) architectures motivated by fixed point iteration schemes for the solution of linear and nonlinear FIEs of the second kind and proved that Fredholm NNs provide accurate solutions. However, there is a little progress on XAI based on neural operator. 

The modeling of financial markets under uncertainty is fundamentally captured through stochastic differential equations, with the Black–Scholes model providing a cornerstone for option pricing by assuming that asset prices follow a continuous geometric Brownian motion \cite{blackscholes1973}. While this model is analytically elegant and highly influential, its assumption of continuous price paths and constant volatility limits its ability to capture abrupt market movements. To overcome these limitations, Robert C. Merton\cite{merton1976option} extended the framework through the Merton jump diffusion model, incorporating both continuous diffusion and discrete jumps via a Poisson process to better reflect real market behavior. Another application, contagion in financial networks refers to the propagation of distress or default from one institution to others through interconnected liabilities, a phenomenon systematically modeled by the Eisenberg–Noe model, which determines clearing payments under mutual obligations. This framework, developed by Larry Eisenberg and Thomas H. Noe\cite{eisenberg2001systemic}, characterizes equilibrium payments using fixed-point theory and ensures existence and uniqueness under mild conditions. Solution methodologies typically involve iterative algorithms to compute the clearing vector, along with extensions incorporating stochastic shocks, fire sales, and network dynamics. 

The application of explainable artificial intelligence-based neural operator approach is used to discuss the SFIEs and SDNNs. First, the fundamentals of stochastic integrals and stochastic differential equations are reviewed to establish the necessary background. The subsequent section introduces both linear and nonlinear stochastic Fredholm integral equations, along with their corresponding neural operator–based explainable artificial intelligence (XAI) frameworks. The construction of these models and their convergence properties are analyzed rigorously from a theoretical perspective. Finally, the last section discusses three significant financial applications like Black-Scholes's model, financial contagion dynamics and Merten jump diffusion model that demonstrate the practical relevance of the proposed approach.
\section{Methedology}
The objective is to apply explainable artificial intelligence (ExAI) to understand SDNNs using SFIE. Some of the theoretical foundations presented below have already appeared in the literature that investigates the relationship between stochastic integral equations (SIEs) and SDNNs. Nevertheless, for the sake of completeness and to maintain a self-contained exposition, we briefly review the essential preliminaries concerning the existence and properties of solutions to SFIEs.\\
Consider stochastic Fredholm integral equations of the second kind\cite{DaPratoZabczyk1992} defined on a complete filtered probability space 
$(\Omega, \mathcal{F}, \{\mathcal{F}_t\}_{t\ge0}, \mathbb{P})$, of the form:

\begin{equation}
	Y(t) = g(t) + \int_D K(t,s) Y(s)\, ds 
	+ \int_D G(t,s) Y(s)\, dW_s, \label{eq:1}
\end{equation}
\noindent where $W_s$ is a standard Brownian motion, $K,G : D\times D \to \mathbb{R}$ are measurable kernels,
$g : D \to \mathbb{R}$ is a deterministic function.
We consider the Hilbert space $
\mathcal{H} := L^2(\Omega; L^2(D)),
$ endowed with the norm

\[
\|Y\|_{\mathcal{H}} 
=
\left(
\mathbb{E}\left[\int_D |Y(t)|^2 dt\right]
\right)^{1/2}.
\]
Here, $\mathcal{H}$ is complete and suitable for studying mean-square solutions of \eqref{eq:1}.
Next, we define the stochastic Fredholm integral operator $\mathcal{S} : \mathcal{H} \to \mathcal{H}$ by ($\mathcal{S}Y)(t)=Y(t)$ and has the mean-square contraction properties
\begin{equation}
	\|\mathcal{S}Y_1 - \mathcal{S}Y_2\|_{\mathcal{H}}
	\le q \|Y_1 - Y_2\|_{\mathcal{H}},\;\;\;\; \text{for some,}\;\; 0 < q < 1.  
	\label{eq:2}
\end{equation}
\noindent If $q = 1$, we say that the operator is mean-square non-expansive. Also, we define the deterministic and stochastic integral operators:
\begin{align}
	(\mathcal{I}_K Y)(t) &= \int_D K(t,s) Y(s) ds, \label{eq:3} \\
	(\mathcal{I}_G Y)(t) &= \int_D G(t,s) Y(s) dW_s. \label{eq:4}
\end{align}
Thus,
\[
\mathcal{S}Y = g + \mathcal{I}_K Y + \mathcal{I}_G Y.
\]
Here, the mean square contraction of $\mathcal{S}$ is proved using Young's inequality on deterministic term, Itô isometry on the stochastic integral term and suitable bound of $K$ and $G$.  For a mean-square contraction $\mathcal{S}$, \eqref{eq:1} admits a unique solution in $\mathcal{H}$ 
via the stochastic Banach fixed-point theorem. In particular, the stochastic Picard iteration:
\begin{equation}
	Y_{n+1} = \mathcal{S} Y_n,
	\label{eq:5}
\end{equation}
converges in mean-square to the unique fixed point solution $Y^\ast$ of \eqref{eq:1}. \\
\begin{proposition}\cite{Protter2005}
Consider the stochastic Fredholm integral equation (\eqref{eq:1}) defined by a 
mean-square non-expansive operator $\mathcal{S}$ on $\mathcal{H}$. 
Furthermore, consider a sequence $\{\kappa_n\}$, $\kappa_n \in (0,1]$, 
such that $
\sum_{n=0}^{\infty} \kappa_n (1-\kappa_n) = \infty.$ Then the iterative scheme

\begin{equation}
	Y_{n+1}(t)
	=
	Y_n(t)
	+
	\kappa_n \big( \mathcal{S}Y_n(t) - Y_n(t) \big)
	=
	(1-\kappa_n) Y_n(t)
	+
	\kappa_n \mathcal{S}Y_n(t),
	\label{eq:7}
\end{equation}
\noindent with $Y_0(t) = g(t)$, converges in mean-square to the fixed point solution 
$Y^\ast(t)$ of the stochastic Fredholm integral equation.	
\end{proposition}

\subsection{Stochastic FIEs and Deep Neural Networks}
Here, we explore the relationship between SFIEs and SDNNs. Also, we will discuss the construction of SDNN, which is motivated by the stochastic fixed-point iteration framework. The main idea is that the solution of an SFIE can be obtained through successive applications of an integral operator, and this iterative structure naturally resembles the layered architecture of a neural network. In this way, we establish a connection between the mathematical theory of SFIEs and DNNs, thereby creating an explainable and interpretable SDNN-based framework for the stochastic numerical approximation of SFIE solutions.\\
\textbf{Definition} (Stochastic $M$-layer fixed point estimate).  
Consider a stochastic process $Y$ satisfying the SFIE \eqref{eq:1} and a discretized grid 
$Z = \{t_1,\dots,t_n\}$, with $t_i$ be the grid points and $n\in N$. Also, consider the discretized version of the stochastic integral operator:
\begin{align}
	(\mathcal{S}_Z Y_n)(t_i)
	&:= 
	\kappa_n g(t_i)
	+ \sum_{j=1}^{N} Y_n(t_j)
	\Big(
	K(t_i,t_j)\kappa_n \Delta t
	+ (1-\kappa_n)\mathbf{\delta}_{ij}
	\Big)
	\nonumber \\
	&\quad
	+ \sum_{j=1}^{N}
	Y_n(t_j)
	G(t_i,t_j)\kappa_n \Delta W_j,
	\label{eq:3.1}
\end{align}
where $\Delta W_j$ are Brownian increments. Also, we define the stochastic $M$-layer fixed point estimate as

\begin{equation}
	Y_M(t)
	=
	\mathcal{S}_Z
	\Big(
	\mathcal{S}_Z(
	\cdots
	\mathcal{S}_Z(g)
	)
	\Big),
	\label{eq:3.2}
\end{equation}
where the operator is applied $M$ times. Next, we propose the lemma associated with neural network realization of stochastic Krasnosel'skii--Mann iteration using (\cite{KloedenPlaten1992}, \cite{Krasnoselskii1955},  \cite{Mann1953}). \\
\begin{lemma}
Define $ X_n=(X_n(t_1),\ldots,X_n(t_N))^\top \in \mathbb{R}^N.$
Then the $M$-step stochastic Krasnosel' skii-Mann iteration associated with the discretized stochastic Fredholm equation can be written as
\begin{equation}
Y_{n+1}=WY_n+b ,
\label{eq:2.9}
\end{equation}
where $W\in\mathbb{R}^{N\times N}$ and $b\in\mathbb{R}^N$ and are given by
\begin{equation}
W_{ij}
=
\kappa K(t_i,t_j)\Delta t
+
\kappa G(t_i,t_j)\Delta W_j
+
(1-\kappa)\delta_{ij},
\qquad
b_i=\kappa g(t_i), \qquad\kappa\in(0,1).
\label{eq:2.10}
\end{equation}
Consequently, the iteration defines a depth-$M$ fully connected
feedforward architecture of width $N$, in which each layer performs
an affine transformation of the form $Y \mapsto WY+b$.
The approximation at $t\in[0,T]$ is given by
\begin{equation}
Y_M(t)
=
\sum_{j=1}^{N}
\left(
\kappa K(t,t_j)\Delta t
+
\kappa G(t,t_j)\Delta W_j
\right)Y_{M-1}(t_j)
+
\kappa g(t).
\label{eq:2.11}
\end{equation}
\end{lemma}
\begin{proof}
Consider the stochastic Fredholm integral equation
\begin{equation}
Y(t)
=
g(t)
+
\int_0^T K(t,s)Y(s)\,ds
+
\int_0^T G(t,s)Y(s)\,dW_s.
\label{eq:2.12}
\end{equation}
Approximating the integrals using a quadrature rule yields
\begin{equation}
(\mathcal{S}Y)(t_i)
=
g(t_i)
+
\sum_{j=1}^{N}
\left(
K(t_i,t_j)\Delta t
+
G(t_i,t_j)\Delta W_j
\right)Y(t_j)
\label{eq:2.13}
\end{equation}
Using \eqref{eq:7} and \eqref{eq:2.13}, we have
\begin{equation}
Y_{n+1}(t_i)
=
(1-\kappa)Y_n(t_i)
+
\kappa g(t_i)
+
\kappa
\sum_{j=1}^{N}
\left(
K(t_i,t_j)\Delta t
+
G(t_i,t_j)\Delta W_j
\right)Y_n(t_j).
\label{eq:2.14}
\end{equation}

Rearranging terms yields

\begin{equation}
Y_{n+1}(t_i)
=
\kappa g(t_i)
+
\sum_{j=1}^{N}
\left(
K(t_i,t_j)\kappa\Delta t
+
G(t_i,t_j)\kappa\Delta W_j
+
(1-\kappa)\mathbf{1}_{\{i=j\}}
\right)Y_n(t_j).
\label{eq:2.15}
\end{equation}
Let $ Y_n=(Y_n(t_1),\dots,Y_n(t_N))^\top.$
Then the above iteration can be written in matrix form
$ Y_{n+1}=W Y_n + b,$ where the matrix $W$ and vector $b$ are defined by \eqref{eq:2.10}.
\end{proof}
Next, using the above, we will discuss the main result associated with neural network approximation of stochastic Fredholm iterates.
\begin{theorem}
Let $\mathcal{S}:\mathcal{H}\to\mathcal{H}$ denote the stochastic Fredholm operator 
defined in \eqref{eq:2}. Assume that $\mathcal{S}$ is either mean-square 
contractive or mean-square non-expansive and define
\[
\mathcal{Z}^{\ast}
:=
\{Y\in\mathcal{H} \mid Y=\mathcal{S}^{k}g,\; k\ge0\}.
\]
Then for every $Y\in\mathcal{Z}^{\ast}$ and every $\varepsilon>0$ there exists 
a stochastic feedforward neural architecture 
$\mathcal{N}(\cdot;W,B)$ with $M$ hidden layers such that
\[
\|Y-\mathcal{N}(\cdot;W,B)\|_{\mathcal{H}}
\le \varepsilon .
\]

The network parameters $W,B$ are determined by the discretized stochastic 
Fredholm iteration, where the weight matrices incorporate the deterministic 
kernel $K$, the stochastic kernel $G$, and the Brownian increments 
$\Delta W_j$. Each layer of the network implements the affine mapping 
$X\mapsto WX+b$, corresponding to one step of the stochastic 
Krasnosel'skii--Mann iteration.
\end{theorem}
\begin{proof}
Let $Y^\ast \in \mathcal H$ denote the unique mean-square fixed point of the stochastic Fredholm operator $\mathcal S$ defined in \eqref{eq:2}.
Since $\mathcal S$ is assumed to be mean-square contractive, the stochastic Picard iteration
\[
Y_{m+1} = \mathcal S(X_m), \qquad m \ge 0,
\]

converges in $\mathcal H$ to $Y^\ast$. In particular, there exists a constant $0<q<1$ such that

\begin{equation}
\|Y_m - Y^\ast\|_{\mathcal H}
\le
q^m \|X_0 - Y^\ast\|_{\mathcal H}.
\label{eq:2.16}
\end{equation}
Consequently, for any prescribed accuracy $\varepsilon>0$, we can choose $M$ sufficiently large so that
\begin{equation}
\|Y_M - Y^\ast\|_{\mathcal H} \le \varepsilon/2.
\label{eq:2.17}
\end{equation}
Next, by Lemma 3.1, each stochastic Picard iteration can be written in matrix form

\begin{equation}
Y_{m+1}=W Y_m + b,
\label{eq:2.18}
\end{equation}
where the weight matrix $W$ and bias vector $b$ are determined by the discretized stochastic kernels $K$ and $G$ and the Brownian increments $\Delta W_j$. Therefore, the sequence of $M$ iterations corresponds exactly to a feedforward neural network with $M$ hidden layers whose output on the discretized grid is the vector
\[
Y_M(Z)=(Y_M(t_1),\dots,Y_M(t_N))^\top.
\]
Next, we will evaluate the approximation at an arbitrary point $t\in D$.
Define the neural network output

\begin{equation}
	\mathcal N(t)
	=
	W_0 Y_M(Z) + \kappa g(t),
	\label{eq:2.19}
\end{equation}
where
\[
W_0 =
\big(
K(t,t_1)\kappa\Delta t + G(t,t_1)\kappa\Delta W_1,
\dots,
K(t,t_N)\kappa\Delta t + G(t,t_N)\kappa\Delta W_N
\big).
\]
It is easy to see that by construction, $\mathcal N(t)$ corresponds to evaluating the stochastic integral operator applied to the approximate solution $Y_M$.
Using the triangle inequality in $\mathcal H$,

\begin{equation}
\|Y^\ast - \mathcal N(\cdot)\|_{\mathcal H}
\le
\|Y^\ast - Y_M\|_{\mathcal H}
+
\|Y_M - \mathcal N(\cdot)\|_{\mathcal H}.
\label{eq:2.20}
\end{equation}

The first term is controlled by the convergence of the stochastic Picard iteration as shown in \eqref{eq:2}.  The second term corresponds to the discretization error of the stochastic integral evaluation. Proceeding in a similar manner as in (\cite{KloedenPlaten1992}), we have
\[
\|Y_M - \mathcal N(\cdot)\|_{\mathcal H} \to 0
\quad \text{as} \quad \Delta t \to 0.
\]

Hence, by choosing $M$ sufficiently large and the discretization sufficiently fine, we obtain

\[
\|Y^\ast - \mathcal N(\cdot)\|_{\mathcal H} \le \varepsilon.
\]
\end{proof}
Next, we will discuss the mean-square error bound for the stochastic Fredholm neural network.
\begin{proposition}
\label{prop:error}
Let $\mathcal H$ be a Hilbert space and  
$\mathcal S:\mathcal H\to\mathcal H$ be the stochastic Fredholm operator defined in \eqref{eq:1}. 
Assume that $\mathcal S$ is a mean-square $q$-contraction and $Y^\ast$ denote the unique fixed point of $\mathcal S$. Let $Y_M$ be the approximation obtained from the $M$-layer stochastic 
Fredholm neural network corresponding to the discretized stochastic 
Picard iteration. Then the Mean-square approximation error satisfies

\begin{equation}
	\|Y_M-Y^\ast\|_{\mathcal H}
	\le
	\frac{q^M}{1-q}
	\Big(
	C_{\Delta t}
	+
	\|\mathcal S g-g\|_{\mathcal H}
	\Big),
	\label{eq:2.21}
\end{equation}

where $C_{\Delta t}$ denotes the mean-square discretization error of the stochastic integral operator,

\[
C_{\Delta t}
=
\sup_{t\in D}
\left(
\mathbb E
\left|
\int_D K(t,s)g(s)\,ds
-
\sum_{j=1}^{N}K(t,t_j)g(t_j)\Delta t
\right|^2
\right)^{1/2}
\]

\[
+
\sup_{t\in D}
\left(
\mathbb E
\left|
\int_D G(t,s)g(s)\,dW_s
-
\sum_{j=1}^{N}G(t,t_j)g(t_j)\Delta W_j
\right|^2
\right)^{1/2}.
\]
Also, for any For any $\varepsilon>0$, there exists a stochastic Fredholm neural network 
with $M^\ast$ hidden layers such that $
\|Y_{M^\ast}-Y^\ast\|_{\mathcal H}\le\varepsilon,
$ provided that

\begin{equation}
	M^\ast
	\ge
	\frac{
		\ln\!\left(
		\dfrac{\varepsilon(1-q)}
		{C_{\Delta t}+\|\mathcal S g-g\|_{\mathcal H}}
		\right)
	}{\ln q}.
	\label{eq:2.23}
\end{equation}
\end{proposition}
\begin{proof}
Let $\mathcal S_{\Delta t}$ denote the discretized stochastic operator 
obtained by approximating the integrals in \eqref{eq:2} via quadrature and 
Brownian increments:
\begin{equation}
(\mathcal S_{\Delta t}Y)(t)
=
g(t)
+
\sum_{j=1}^{N}
\Big(
K(t,t_j)\Delta t
+
G(t,t_j)\Delta W_j
\Big)Y(t_j).
\label{eq:2.24}
\end{equation}
and the associated neural network iteration is
$Y_{m+1}=\mathcal S_{\Delta t}(Y_m).$ Define the discretization error
$E_{\Delta t}(Y)
=\mathcal S(Y)-\mathcal S_{\Delta t}(Y).$
Using \eqref{eq:2.24}. Using the mean-square convergence results for stochastic quadrature and Euler–Maruyama discretization, we have

\[
\|E_{\Delta t}(g)\|_{\mathcal H}
\le
C_{\Delta t}.
\]
Now, \[
Y_{m+1}-Y^\ast
=
\mathcal S_{\Delta t}(Y_m)-\mathcal S(Y^\ast)=\big(S_{\Delta t}(Y_m)-\mathcal S(Y_m)\big)+\big(S(Y_m)-\mathcal S(Y^\ast)\big)
\]
Using the $\mathcal H$-norm and applying the triangle inequality yields
\begin{equation}
    \|Y_{m+1}-Y^\ast\|_{\mathcal H}
\le
\|\mathcal S_{\Delta t}(Y_m)-\mathcal S(X_m)\|_{\mathcal H}
+
\|\mathcal S(X_m)-\mathcal S(X^\ast)\|_{\mathcal H}.
\label{eq:2.25}
\end{equation}
Using the contraction property of $\mathcal S$,

\[
\|\mathcal S(X_m)-\mathcal S(X^\ast)\|_{\mathcal H}
\le
q\|Y_m-X^\ast\|_{\mathcal H}.
\]

Furthermore, the discretization term satisfies

\[
\|\mathcal S_{\Delta t}(Y_m)-\mathcal S(Y_m)\|_{\mathcal H}
\le
C_{\Delta t}+\|\mathcal S g-g\|_{\mathcal H}.
\]
Therefore, from \eqref{eq:2.25}, we have 

\[
\|Y_{m+1}-Y^\ast\|_{\mathcal H}
\le
q\|Y_m-X^\ast\|_{\mathcal H}
+
C_{\Delta t}
+
\|\mathcal S g-g\|_{\mathcal H}.
\]
\[
\implies\|Y_M-Y^\ast\|_{\mathcal H}
\le
q^M\|Y_0-Y^\ast\|_{\mathcal H}
+
\sum_{k=0}^{M-1} q^k
\left(
C_{\Delta t}
+
\|\mathcal S g-g\|_{\mathcal H}
\right).
\]
From the above, we have

\[
\|Y_M-X^\ast\|_{\mathcal H}
\le
\frac{q^M}{1-q}
\left(
C_{\Delta t}
+
\|\mathcal S g-g\|_{\mathcal H}
\right).
\]
Choosing  $M$ such that
\[
\frac{q^M}{1-q}
\left(
C_{\Delta t}
+
\|\mathcal S g-g\|_{\mathcal H}
\right)
\le
\varepsilon.
\]
Solving for $M$ yields

\[
M
\ge
\frac{
\ln\!\left(
\dfrac{\varepsilon(1-q)}
{C_{\Delta t}+\|\mathcal S g-g\|_{\mathcal H}}
\right)
}{\ln q}.
\]
\end{proof}
All the above was regarding linear stochastic Fredholm integral equation and Fredholm neural network. In the next section we will discuss the nonlinear stochastic Fredholm integral equation and the associated nonlinear deep nonlinear stochastic Fredholm neural network (NSFNN). 
\subsection{Nonlinear Stochastic Fredholm Integral Equation(NSFIE)}
Here, we will discuss the NSFIE and will construct the associated neural operator based deep nonlinear stochastic Fredholm neural network (NSFNN). Consider the NSFIE 

\begin{equation}
\mathcal S(Y)=f(t)+\int_0^T K(t,s)\Phi(Y(s))ds
+\int_0^T G(t,s)Y(s)dW(s),
\end{equation}

where $f \in L^2([0,T])$, 
$K(t,s)$ is a deterministic Fredholm kernel, $G(t,s)$ is a stochastic kernel with $K, G \in L^\infty([0,T]^2)$, $W(t)$ denotes Brownian motion, $\Phi: \mathbb{R} \to \mathbb{R}$ is Lipschitz continuous:
\[
|\Phi(x) - \Phi(y)| \le L_\Phi |x-y|.
\] Cosider the Deep Stochastic Fredholm Neural Network (DSFNN)
\begin{equation}
Y^{(k+1)} = \mathcal{S}(Y^{(k)}),
\end{equation}
where 
\begin{equation}
Y^{(k+1)}(t)=f(t)
+\int_0^T K(t,s)\Phi(Y^{(k)}(s))ds
+\int_0^T G(t,s)Y^{(k)}(s)dW(s).
\end{equation}

\begin{proposition}
Let $\mathcal{S}$ denote the stochastic Fredholm operator

\begin{equation}
\mathcal{S}(Y)=f + K\Phi(Y) + G(Y\odot dW).
\end{equation}

Assume the activation function $\Phi$ is Lipschitz continuous with constant $L_\Phi$. If $
\|K\|L_\Phi + \|V\|L_\Phi < 1,$
then the operator $\mathcal{S}$ is a contraction in mean square and the sequence
$
Y^{(k)} \rightarrow Y
$ converges to the unique solution of the stochastic integral equation.
\end{proposition}
\begin{proof}
   Let $Y_1, Y_2 \in \mathcal{H}$ and consider
\[
S(Y_1)(t) - S(Y_2)(t)
\]
\[
= \int_0^T K(t,s)\left[\Phi(Y_1(s)) - \Phi(Y_2(s))\right] ds
+ \int_0^T G(t,s)\left[Y_1(s) - Y_2(s)\right] dW(s).
\] 
Let
\[
I_1(t) = \int_0^T K(t,s)\left[\Phi(Y_1(s)) - \Phi(Y_2(s))\right] ds,\;
I_2(t) = \int_0^T G(t,s)\Delta Y(s)\, dW(s)
\]
where $\Delta Y = Y_1 - Y_2$ and $S(Y_i)(t)=I_i(t)$ for $i=1,2$.
Using Lipschitz continuity we have 
\[
|I_1(t)| \le L_\Phi \int_0^T |K(t,s)| |\Delta Y(s)| ds.
\]
Next, using Cauchy–Schwarz inequality and 
Integrating over $t$ and taking the expectation, we will get
\[
\mathbb{E}\int_0^T |I_1(t)|^2 dt
\le L_\Phi^2 \|K\|^2 \mathbb{E}\int_0^T |\Delta Y(s)|^2 ds.
\]
Using Itô isometry and integrating over $T$, we have
\[
\mathbb{E}\int_0^T |I_2(t)|^2 dt
\le \|G\|^2 \mathbb{E}\int_0^T |\Delta Y(s)|^2 ds.
\]
Now,
\[
\|S(Y_1) - S(Y_2)\|
\le \sqrt{2}\left( L_\Phi \|K\| + \|G\| \right)\|Y_1 - Y_2\|.
\]
If $
\sqrt{2}\left( L_\Phi \|K\| + \|G\| \right) < 1,$
then $S$ is a contraction.
By Banach fixed point theorem, $S$ admits a unique fixed point $Y$
and the iteration $Y^{(k+1)} = S(Y^{(k)})$ converges in mean square. Hence,
\[
Y^{(k)} \to Y \quad \text{in } L^2(\Omega \times [0,T]).
\]
\end{proof}
\begin{algorithm}[H]
\caption{Deep Stochastic Fredholm Neural Network}
\begin{enumerate}
\item Discretize the time grid $\{t_i\}_{i=1}^N$.
\item Construct kernel matrices
\[
K_{ij}=K(t_i,t_j)\Delta t, \quad 
V_{ij}=V(t_i,t_j)\Delta t.
\]
\item Simulate Brownian increments
\[
\Delta W_i \sim \mathcal{N}(0,\Delta t).
\]
\item Initialize
\[
X^0 = f.
\]
\item For $k=0,\dots,K$
\[
X^{k+1} = f + K\Phi(X^k) + G(X^k\odot \Delta W)
\]
\item Stop when
\[
\|X^{k+1}-X^k\| < \varepsilon .
\]
\end{enumerate}
\end{algorithm}
The computational details of nonlinear SFNN are given in the algorithm 3.1. Next, we further gneralize the NSFNN to the nonlinear Volterra-Fredholm Stochastic Model. 

\subsection{Volterra--Fredholm Stochastic Model}
We consider the stochastic Volterra--Fredholm equation:
\begin{equation}
Y(t) = f(t) + \int_0^t V(t,s)\Phi(Y(s))\,ds 
+ \int_0^T K(t,s)\Phi(Y(s))\,ds 
+ \int_0^t G(t,s)Y(s)\,dW(s),\label{3.5}
\end{equation}
with $V(t,s)$ is the Volterra kernel (memory effect), $K(t,s)$ is the Fredholm kernel (global interaction),
$G(t,s)$ is the stochastic kernel,
$\Phi$ is a nonlinear function. Define operator $S$:
\begin{equation}
(SY)(t) = f(t) + \int_0^t V(t,s)\Phi(Y(s))ds 
+ \int_0^T K(t,s)\Phi(Y(s))ds 
+ \int_0^t G(t,s)Y(s)dW(s). \label{3.6}
\end{equation}
\begin{theorem}
    Let $\Phi$ be Lipschitz and the kernel of $\Phi$ satisfy
$\|V\| + \|K\| + \|G\| < \infty$ and $ L(\|V\| + \|K\|) + \|G\| < 1,$
then the equation \eqref{3.6} admits a unique solution in $H$.
\end{theorem}
\begin{proof}
Let $Y_1, Y_2 \in H$. Then
\[
S(Y_1) - S(Y_2)
\]
\[
= \int_0^t V(t,s)[\Phi(Y_1)-\Phi(Y_2)]ds 
+ \int_0^T K(t,s)[\Phi(Y_1)-\Phi(X_2)]ds 
+ \int_0^t G(t,s)(Y_1-Y_2)dW(s).
\]
Using the expectation and norm and kernel bounds $C_1$ and $C_2$ and It\^o isometry, we have
\[
\|S(Y_1) - S(Y_2)\|
\le \big[ L(\|V\| + \|K\|) + \|G\| \big] \|Y_1 - Y_2\|\le \|Y_1 - Y_2\|,
\]
Since $H$ is complete and $S$ is a contraction, there exists a unique $X \in H$ such that
\[
S(Y) = Y.
\]
\end{proof}
\begin{theorem}
Let $Y_\theta(t)$ be a neural network approximation of $Y(t)$. Define the residual and loss function respectively as
\[
R_\theta(t) = Y_\theta(t) - f(t) - \int_0^t V(t,s)\Phi_s\,ds
\;-\; \int_0^t K(t,s)\Phi_s\,ds
\;-\; \int_0^t G(t,s)Y_\theta(s)\,dW(s),
\]
$\mathcal{L}(\theta) = \mathbb{E}\!\left[\displaystyle\int_0^T |R_\theta(t)|^2\,dt\right]$,
 with $\Phi_s = \Phi(Y_\theta(s))$. \quad
Then we have
\begin{enumerate}
  \item If $\mathcal{L}(\theta) = 0$, then $Y_\theta(t) = Y(t) \ \mathbb{P}\text{-a.s.}.$
  \item If $\mathcal{L}(\theta) \to 0$, then $Y_\theta \to Y$ in $L^2(\Omega \times [0,T])$.
\end{enumerate}
\end{theorem}
 
\begin{proof}
(i) Given $\mathcal{L}(\theta) = 0$ and we have to show $Y_\theta(t) = Y(t)$ a.s.
 Since $|R_\theta(t)|^2 \geq 0$ is measurable and non-negative, so
\[
\mathcal{L}(\theta) = \mathbb{E}\!\left[\int_0^T |R_\theta(t)|^2\,dt\right] = 0
\]
Hence
\[
R_\theta(t) = 0 \quad \text{for a.e. } t \in [0,T].
\]
 
Since $Y_\theta(t)$ satisfies the governing equation so setting $R_\theta(t) = 0$ gives
\[
Y_\theta(t) = f(t)
  + \int_0^t V(t,s)\,\Phi(Y_\theta(s))\,ds
  + \int_0^t K(t,s)\,\Phi(Y_\theta(s))\,ds
  + \int_0^t G(t,s)\,Y_\theta(s)\,dW(s),
\]
which is exactly the stochastic integro-differential equation satisfied by $Y(t)$.
Under the assumptions Lipschitz continuity of $\Phi$, $V$, $K$, $G$ and square-integrability of the coefficients, the equation admits a unique solution in $L^2(\Omega \times [0,T])$ by a Picard fixed-point theorem. Since $Y_\theta(t)$ and $Y(t)$ satisfy the same equation with the same initial data,
\[
Y_\theta(t) = Y(t) \quad \mathbb{P}\text{-a.s.}
\]
 
\bigskip
\textbf{(ii)}
Let $e_\theta(t) = Y_\theta(t) - Y(t)$. Hence, 
\begin{align*}
e_\theta(t) &= R_\theta(t)
  + \int_0^t V(t,s)\bigl[\Phi(Y_\theta(s)) - \Phi(Y(s))\bigr]\,ds \\
  &\quad + \int_0^t K(t,s)\bigl[\Phi(Y_\theta(s)) - \Phi(Y(s))\bigr]\,ds
  + \int_0^t G(t,s)\,e_\theta(s)\,dW(s).
\end{align*}
Using the inequality $(a+b+c+d)^2 \leq 4(a^2+b^2+c^2+d^2)$, the Cauchy--Schwarz inequality, the Lipschitz condition $|\Phi(x)-\Phi(y)| \leq L|x-y|$, boundedness of the kernels $V$, $K$, $G$, and the {It\^{o} isometry} applied to the stochastic integral, we obtain
\[
\mathbb{E}|e_\theta(t)|^2
\;\leq\; C\,\mathbb{E}|R_\theta(t)|^2 + C\int_0^t \mathbb{E}|e_\theta(s)|^2\,ds,
\]
for some constant $C > 0$ depending on $L$, $T$, and the kernel bounds.
 
Integrating over $[0,T]$
\[
\int_0^T \mathbb{E}|e_\theta(t)|^2\,dt
\;\leq\; C\int_0^T \mathbb{E}|R_\theta(t)|^2\,dt
  + C\int_0^T\!\int_0^t \mathbb{E}|e_\theta(s)|^2\,ds\,dt.
\]
and applying {Gr\"{o}nwall's inequality} yields
\[
\|e_\theta\|^2_{L^2(\Omega\times[0,T])}
\;\leq\; C\,\mathcal{L}(\theta)\,e^{CT}.
\]
Since $\mathcal{L}(\theta) \to 0$, the right-hand side tends to zero, giving
\[
\|Y_\theta - Y\|_{L^2(\Omega\times[0,T])} \;\longrightarrow\; 0,
\]
which gives $Y_\theta \to Y$ in $L^2(\Omega \times [0,T])$.
\end{proof}
\section {Applications to Financial Integral Equations}

In this section we will discuss the application of SFIE, SFDNN, NSFIE, NSFDNN and Voltera stochastic deep neural networks. Also, we will discuss the acuaracy and the convergence of the solutions obtained through each method. Again, we will discuss the accuracy and usefulness of the constructed neural operators.  
\subsection {Solution of Black Schole's equation}
Consider an asset price $S_t$ following a geometric Brownian motion
\[
dS_t = rS_t\,dt + \sigma S_t\,dW_t,
\]
where $r$ is the risk-free rate, $\sigma$ the volatility and $W_t$ a
standard Brownian motion. Let $V(S)$ denote the price of a down-and-out barrier option.
Under risk-neutral valuation, $V(S)$ satisfies the Black-Schole's equation
\begin{equation}
\frac{1}{2}\sigma^2 S^2 V''(S) + rS V'(S) - rV(S) = 0,
\qquad H < S < K,
\label{eq:4.1}
\end{equation}
with boundary conditions
\begin{equation}
V(H) = 0, \qquad V(K) = K-H .
\label{eq:bvp_bc}
\end{equation}

The condition $V(H)=0$ reflects the barrier feature: if the asset price
touches the barrier $H$, the option becomes worthless. The condition
$V(K)=K-H$ represents the payoff at the strike level. Using the change of variable $x = \ln S, \quad \tau = T - t,$ we will reduce the Black Scole's equation to a simple parabolic PDE(heat equation). Using $V(S,t) = e^{-r\tau} u(x,\tau)$, $u(x,\tau) = e^{ax + b\tau} w(x,\tau),$ $a = -\frac{r - \frac{1}{2}\sigma^2}{\sigma^2}, \quad
b = -\frac{1}{2}\sigma^2 a^2 - a\left(r - \frac{1}{2}\sigma^2\right)$, equation \eqref{eq:4.1} reduces to the heat equation
\begin{equation}
\frac{\partial w}{\partial \tau}
= \frac{1}{2}\sigma^2 \frac{\partial^2 w}{\partial x^2}.
\end{equation}
The Green's function for the above is derived as 
\begin{equation}
G(x,\tau;\xi)
=
\frac{1}{\sqrt{2\pi\sigma^2 \tau}}
\exp\left(
-\frac{(x-\xi)^2}{2\sigma^2 \tau}
\right).
\end{equation}
Hence, 
\begin{equation}
w(x,\tau)
=
\int_{-\infty}^{\infty}
G(x,\tau;\xi)\, f(\xi)\, d\xi,
\end{equation}
where
$f(\xi) = e^{\alpha \xi}\max(e^\xi - K,0).$
Since $G(x,\tau;\xi)$ is a Gaussian density, we write $X_\tau = x + \sigma W_\tau,$
where $W_\tau$ is a Wiener process. Then $w(x,\tau)
=\mathbb{E}[f(X_\tau)].$
Next, define the stochastic kernel
\begin{equation}
K(x,\tau;\xi,\omega)
=
\frac{1}{\sqrt{2\pi\sigma^2 \tau}}
\exp\left(
-\frac{(x - \xi - \sigma W_\tau(\omega))^2}{2\sigma^2 \tau}
\right).
\end{equation}

Then
\begin{equation}
w(x,\tau,\omega)
=
\int_{-\infty}^{\infty}
K(x,\tau;\xi,\omega)\, f(\xi)\, d\xi.
\end{equation}
Next, the SFIE of second kind is obtained using $x = \ln S, \quad \xi = \ln y$ and is given by
\begin{equation}
V(S,t,\omega)
=
\max(S-K,0)
+
\lambda \int_0^\infty
K(S,t;y,\omega)\,
V(y,t,\omega)\, dy.
\end{equation}
where
\begin{align}
K(S,t;y,\omega)\nonumber
&=
\frac{1}{y\sigma\sqrt{2\pi (T-t)}} \\
&\quad \times
\exp\left(
-\frac{
\left(
\ln\frac{S}{y}
-
\left(r-\frac{\sigma^2}{2}\right)(T-t)
-
\sigma W_{T-t}(\omega)
\right)^2
}{2\sigma^2 (T-t)}
\right).
\end{align}
In the finite interval the SFIE can be rewritten as 
\begin{equation}
\nu(S) = f(S) + \int_H^K K(S,t)\nu(t)\,dt,
\label{eq:fredholm_finance}
\end{equation}
where
\[
f(S) = -\frac{2}{\sigma^2 S^2}
\left( rV(S) - rS V'(S) \right),
\]
and the kernel is given by
\[
K(S,t) =
\begin{cases}
t(K-S), & H \le t \le S, \\
S(K-t), & S \le t \le K .
\end{cases}
\]


Let $\{S_i\}_{i=1}^N$ be a discretization of the interval $[H,K]$ with
step size $\Delta S$. The Fredholm equation \eqref{eq:fredholm_finance}
can be approximated by
\[
\nu_i = f(S_i) + \sum_{j=1}^N K(S_i,S_j)\nu_j \Delta S .
\]

Following the Fredholm Neural Network framework, the weights and biases
of the network are defined as
\[
W_{ij} = K(S_i,S_j)\Delta S,
\qquad
b_i = f(S_i).
\]

The SFNN iteration is then given by
\[
\nu^{(k+1)} = b + W \nu^{(k)},
\]
which corresponds to a feedforward neural network with $k$ hidden layers.

Once $\nu(S)$ is obtained, the option price can be reconstructed via
\[
V(S) = \int_H^S \int_H^x \nu(t)\,dt\,dx + C_1 S + C_2,
\]
where the constants $C_1$ and $C_2$ are determined from the boundary
conditions \eqref{eq:bvp_bc}.

In Fig(1), the solution obtained through Black-Scholes, Green's integral, SFIE and SFNN are plotted. The Green's function integral achieves near-machine-precision accuracy
($\sim\!10^{-15}$) in pricing European call options, with the SFIE fixed-point
iteration.
The time-slice panel further validates the framework,
showing monotonic convergence of Black-Scholes price curves toward the intrinsic
payoff $\max(S - K,\, 0)$ as expiry approaches. By contrast, the SFNN Galerkin
projection onto a degree-$15$ polynomial basis incurs a global relative error of
$\approx 1.85 \times 10^{-3}$, roughly three orders of magnitude larger than the
quadrature-based methods.

\begin{figure}
    \centering
    \includegraphics[width=0.8\linewidth]{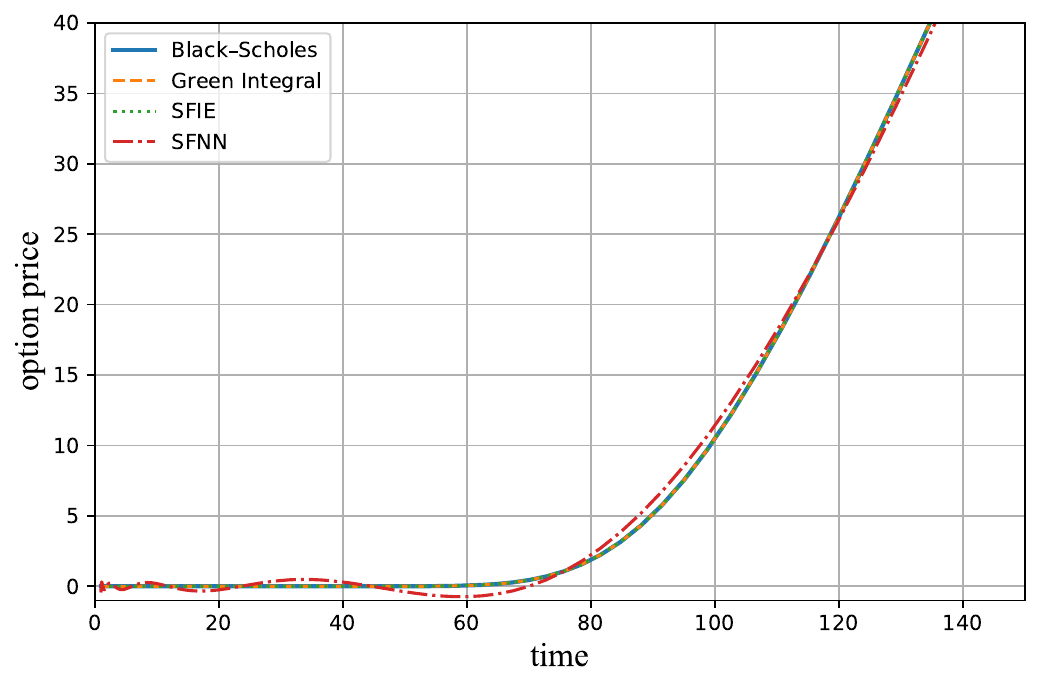}
    \caption{Comparision SFIE, SFNN and Exact}
    \label{fig:placeholder}
\end{figure}
Next, we will discuss another application of contagion dynamics in the bank networks.

\subsection{Neural operator approach of financial contagion dynamics in interbank networks}
Here, we will discuss another important class of applications called financial contagion. Financial contagion is the propagation of financial distress
across interconnected banks through interbank liabilities and
exposures. A widely studied framework for systemic risk analysis is
the Eisenberg--Noe model \cite{eisenberg2001systemic}, which describes how shocks propagate
through financial networks. Consider a financial system consisting of $N_b$ banks. Let $y_i(t)$
denote the distress level of bank $i$ at time $t$, where $0 \le y_i(t) \le 1.$ The dynamics of contagion can be modeled by the integral equation
\begin{equation}
y_i(t)
=
f_i(t)
+
\sum_{j=1}^{N_b}
a_{ij}
\int_0^t
K(t,s)y_j(s)\,ds ,
\label{eq:5.1}
\end{equation}
where $f_i(t)$ represents external financial shocks to bank $i$, $a_{ij}$ denotes the exposure of bank $i$ to bank $j$, $K(t,s)$ is a contagion kernel describing how past distress influences the present state. The matrix $A=(a_{ij})$ represents the interbank exposure network. Let's assume the exponential decaying kernel as
\begin{equation}
K(t,s)
=
\beta e^{-\gamma (t-s)}, \qquad s\le t,
\label{eq:5.2}
\end{equation}
where $\beta$ denotes the contagion strength, $\gamma$ represents the recovery or stabilization rate due to regulatory intervention or capital injections. Using \eqref{eq:5.1} and \eqref{eq:5.2}, we have
\begin{equation}
y_i(t)
=
f_i(t)
+
\sum_{j=1}^{N_b}
a_{ij}
\int_0^t
\beta e^{-\gamma (t-s)}y_j(s)\,ds .
\label{eq:5.4}
\end{equation}
Now, the Fredholm integral form of \eqref{eq:5.4} is given by 
\begin{equation}
y_i(t)
=
f_i(t)
+
\sum_{j=1}^{N_b}
\int_0^T
K_{ij}(t,s)y_j(s)\,ds ,
\label{eq:5.41}
\end{equation}
where
\[
K_{ij}(t,s)
=
\begin{cases}
a_{ij}\beta e^{-\gamma (t-s)}, & s \le t, \\
0, & s>t.
\end{cases}
\]
Let $t_1,\dots,t_N$ be a uniform discretization of the interval $[0,T]$ with step size $\Delta t$.
The integral equation \eqref{eq:5.41} becomes
\begin{equation}
y_{i,k}
=
f_i(t_k)
+
\sum_{j=1}^{N_b}
\sum_{l=1}^{N}
K_{ij}(t_k,t_l)y_{j,l}\Delta t .
\label{eq:5.5}
\end{equation}
Then the system can be written compactly as $y = b + W y,$ where the block matrix and the vector $b$ are given by $W_{(i,k),(j,l)}=K_{ij}(t_k,t_l)\Delta t,$ $b_{i,k} = f_i(t_k).$ Proceeding in a similar manner and using the initial guess $x^{(0)}$, the SFNN iteration is given by
\begin{equation}
y^{(k+1)} = b + W y^{(k)}, \qquad k=0,1,2,\dots
\label{eq:5.6}
\end{equation}
where $W$ and $b$ represent contagion propagation through interbank exposures and external financial shock respectively. Each iteration corresponds to a neural network layer,
and the fixed point of \eqref{eq:5.6} gives the equilibrium contagion level. The above iteration converges if the spectral radius of the kernel matrix satisfies $\rho(W) < 1.$

The left panel illustrates financial contagion dynamics across five interconnected
banks, each initialized at a distress level of $d = 0.10$ at $t = 0$, with all
trajectories decaying monotonically toward near-zero distress over the interval
$t \in [0, 10]$. Banks~1 through~4 exhibit tightly clustered decay profiles,
reflecting similar contagion exposure and balance-sheet coupling, whereas Bank~5
undergoes a markedly faster initial decline, suggesting a lower systemic
 interconnectedness or stronger absorptive capacity. The near-linear decay on the semilogarithmic scale
confirms exponential convergence of the SFNN iteration with respect to network
depth, validating the framework as a high-fidelity solver for contagion dynamics
with error comparable to direct quadrature methods.

\begin{figure}
    \centering
    \includegraphics[width=0.9\linewidth]{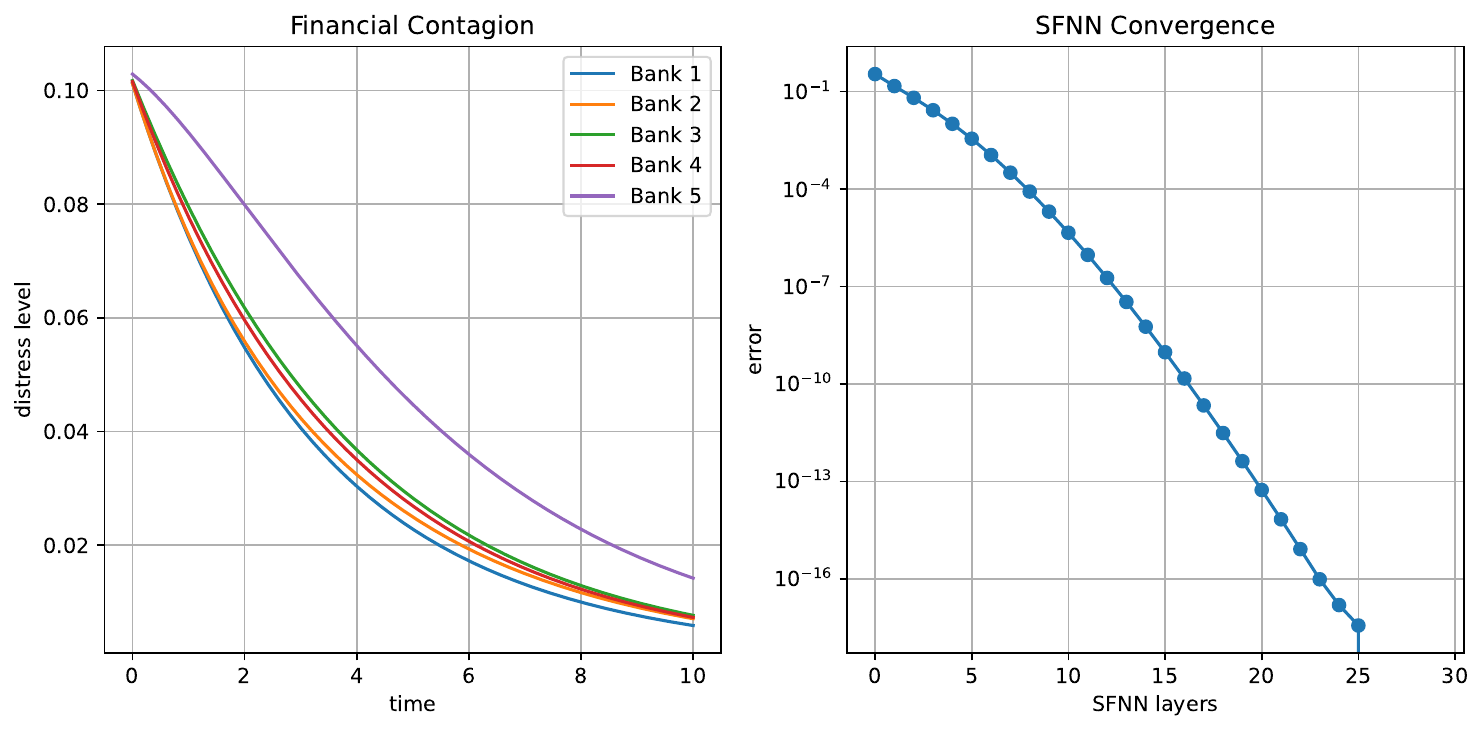}
    \caption{Financial contagion dynamics using SFNN and convergence}
    \label{fig:placeholder}
\end{figure}
Next, we will discuss another important application Merten jump discontinuity.
\subsection{Solution of Merten Jump Diffision Equation: A neural operator approach}
The Merton jump diffusion model \cite{merton1976option} describes the evolution of an asset price $S(t)$ as:

\begin{equation}
dS(t) = \mu S(t)\,dt + \sigma S(t)\,dW(t) + S(t^-)(J - 1)\,dN(t),\label{eq:6.1}
\end{equation}
where $\mu$ is the drift coefficient, $\sigma$ is the volatility, $W(t)$ is a standard Brownian motion, $N(t)$ is a Poisson process with intensity $\lambda$, $J$ is a random jump multiplier, $S(t^-)$ denotes the left limit of $S(t)$. Integrating both sides from $0$ to $t$, we obtain:
\begin{equation}
S(t) = S(0) + \int_0^t \mu S(s)\,ds + \int_0^t \sigma S(s)\,dW(s) + \int_0^t S(s^-)(J-1)\,dN(s).\label{eq:6.2}
\end{equation}
Using the compensated Poisson process:$\tilde{N}(t) = N(t) - \lambda t$ and $dN(t) = d\tilde{N}(t) + \lambda dt$ and substituting into the jump term we have
\begin{align}
\int_0^t S(s^-)(J-1)\,dN(s)
= \int_0^t S(s^-)(J-1)\,d\tilde{N}(s) 
\quad + \lambda \int_0^t S(s)\mathbb{E}[J-1]\,ds.\label{eq:6.3}
\end{align}
Using \eqref{eq:6.1} in \eqref{eq:6.2} and $\tilde{\mu} = \mu + \lambda \mathbb{E}[J-1]$, we have
\begin{equation}
S(t) = S(0) + \int_0^T K(t,s) S(s)\,ds
+ \int_0^t \sigma S(s)\,dW(s)
+ \int_0^t S(s^-)(J-1)\,d\tilde{N}(s).
\end{equation}
with
\begin{equation}
K(t,s) =
\begin{cases}
\tilde{\mu}, & 0 \le s \le t, \\
0, & s > t.
\end{cases}
\end{equation}
The stochastic Fredholm--Volterra neural network is constructed as a fixed-point iteration of the nonlinear stochastic operator. Define the iterative sequence $\{S_\theta^{(k)}(t)\}_{k\geq 0}$ by

\begin{align}
S_\theta^{(k+1)}(t)
&= S_0
+ \int_0^T K_\theta(t,s)\,S_\theta^{(k)}(s)\,ds \\
&\quad + \int_0^t \sigma S_\theta^{(k)}(s)\,dW(s)
+ \int_0^t S_\theta^{(k)}(s^-)(J-1)\,d\tilde{N}(s),
\end{align}

for $k = 0,1,2,\dots$, with an initial guess $S_\theta^{(0)}(t)$ (typically $S_\theta^{(0)}(t)=S_0$).
This defines a stochastic operator iteration
\begin{equation}
S_\theta^{(k+1)} = \mathcal{T}_\theta\big(S_\theta^{(k)}\big).
\end{equation}
Let $\{t_i\}_{i=0}^N$ be a uniform partition of $[0,T]$ with step size $\Delta t$. Then the discrete SFVNN iteration becomes

\begin{align}
S_i^{(k+1)} &= S_0
+ \sum_{j=0}^{N} K_\theta(t_i,t_j)\,S_j^{(k)}\,\Delta t \\
&\quad + \sum_{j=0}^{i-1} \sigma S_j^{(k)}\,\Delta W_j
+ \sum_{j=0}^{i-1} S_j^{(k)}(J_j - 1)\,\Delta \tilde{N}_j.
\end{align}
Here, The neural network is used exclusively to parameterize the Fredholm kernel $K_\theta(t,s)$ and the model enforces the intrinsic structure of the stochastic Fredholm--Volterra equation at every iteration.
Thus, the SFVNN is defined by:
\begin{equation}
S_\theta = \mathcal{T}_\theta(S_\theta),
\end{equation}
where $\mathcal{T}_\theta$ is a learned stochastic integral operator.

In Fig(3), the SFVNN fixed-point residual $\|S^{(k+1)} - S^{(k)}\| / \|S^{(k)}\|$ decays
monotonically from $\sim\!10^{-1}$ at the first outer iteration to
$\approx 10^{-13}$ by iteration $k = 12$, crossing the prescribed tolerance
$\varepsilon = 10^{-12}$ between iterations~11 and~12. The near-linear profile
on the semilogarithmic scale confirms superlinear, effectively geometric,
convergence of the fixed-point scheme, validating the contraction property of the
SFVNN operator across successive iterates.

In Fig(4), the Fredholm residual loss decays sharply from $\sim\!10^{-1}$ at the onset of stochastic gradient descent (SGD) training to a smoothed floor of $\approx 10^{-4}$ by step~12{,}000, with a
brief transient spike near step~2{,}000 marking the outer iteration boundary
where the neural network operator is re-initialized for the subsequent fixed-point
sweep. The width-$10$ moving average  confirms steady,
monotonically decreasing trend across all concatenated outer iterations, while the
per-step noise envelope remains bounded, demonstrating stable stochastic
convergence of the neural network operator toward the target Fredholm integral
kernel.

In Fig(5), the fixed-point residual decays exponentially across iterations
$k = 1, \ldots, 12$, dropping from $\sim\!10^{-2}$ to $\approx 10^{-13}$ and
confirming the geometric contraction of the SFVNN scheme, while the neural
network approximation error relative to the true fixed-point solution plateaus at
$\approx 10^{-2}$ after the first iteration and remains essentially flat
thereafter. This decoupling demonstrates that the FP residual is driven to
machine precision by the outer iteration, whereas the irreducible neural network operator
error is bounded below by the expressivity of the chosen network architecture,
constituting the dominant source of approximation error in the scheme.

\begin{figure}
    \centering
    \includegraphics[width=0.9\linewidth]{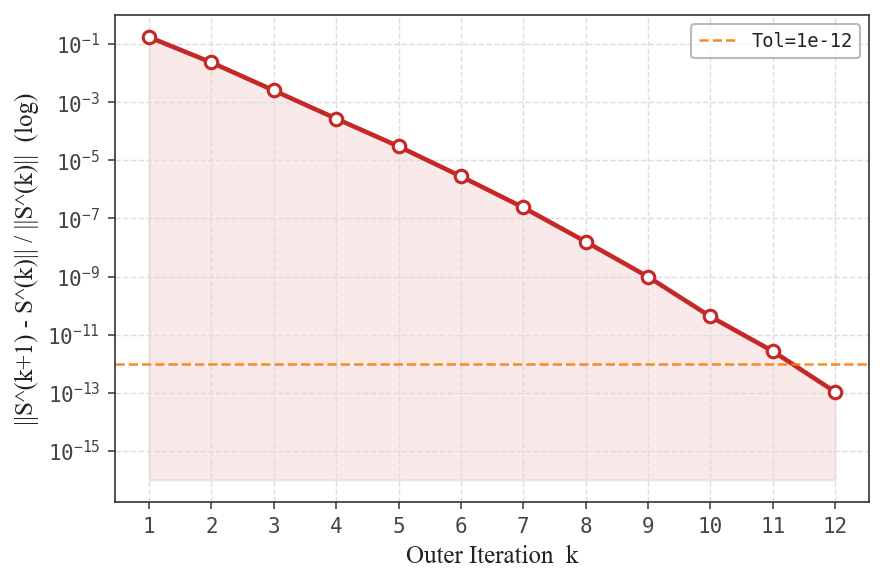}
    \caption{Stochastic Voltera Fredholm Neural Network residual}
    \label{fig3}
\end{figure}
\begin{figure}
    \centering
    \includegraphics[width=0.9\linewidth]{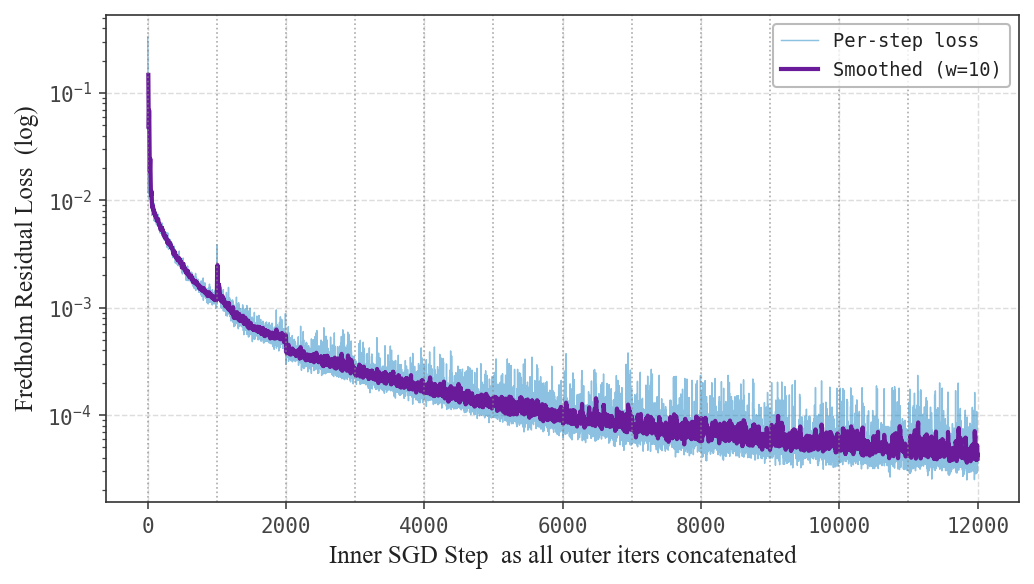}
    \caption{Neural Network Operator Training loss}
    \label{fig4}
\end{figure}
\begin{figure}
    \centering
    \includegraphics[width=0.9\linewidth]{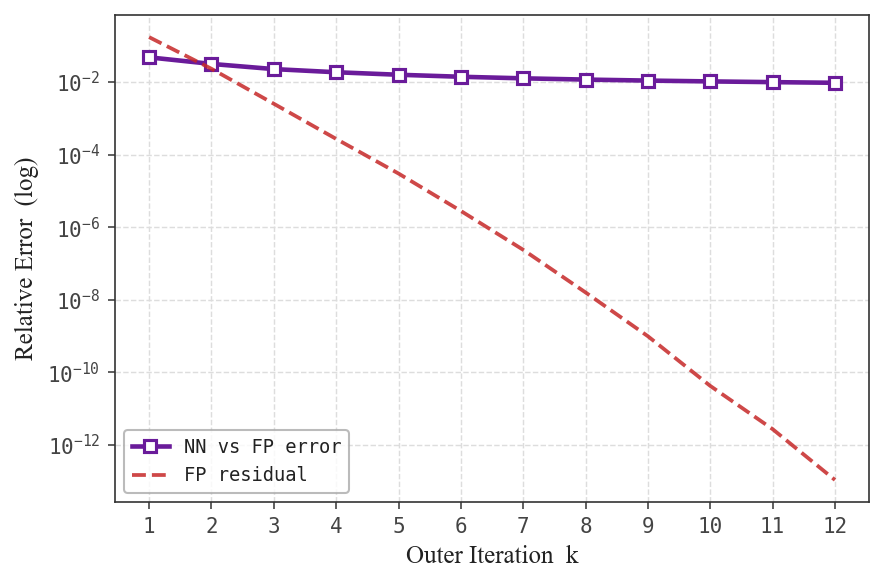}
    \caption{Neural Network Error Vs fixed point residual}
    \label{fig5}
\end{figure}
\section{Conclusion}
The neural operator-based explainable artificial intelligence is used to discuss the stochastic Fredholm integral equations (SFIEs) and stochastic
deep neural networks (SDNNs). Using the linearized framework, the connection between SFIE and SDNN is established. SFIE and SDNN are used to solve the Black-Scholes equation, the contagion dynamics of financial networks, and the Merten jump diffusion equation and we found that SDNN provides accurate solutions. SFIE and SFNN are used to solve the forward problems, whereas the inverse problem is complicated due to the instability nature of the associated hyperbolic partial differential equation and can be discussed as a separate problem. 

\section*{Acknowledgments}
This work is supported by the Department of Science and Technology, Government of India, under
Grant SR/FST/MS-II/2023/139(C)-VIT Vellore.
\biboptions{authoryear}
\bibliographystyle{elsarticle-harv}
\bibliography{references}


 \end{document}